%% file: Fitting.tex
\documentclass[12pt]{amsart}

\usepackage[pdfauthor   = {Mohamed\ Barakat\ and\ Lukas\ Kuehne},
            pdftitle    = {Computing\ the\ nonfree\ locus\ of\ the\ moduli\ space\ of\ arrangements\ and\ Terao's\ freeness\ conjecture},
            pdfsubject  = {},
            pdfkeywords = {rank 3 simple matroids;\ integrally splitting characteristic polynomial;\ Terao's freeness conjecture;\ noSQL database;\ ArangoDB},
            bookmarks=true,
            bookmarksopen=true,
            pagebackref=true,
            hyperindex=true,
            colorlinks=true,
            linkcolor=blue,
            citecolor=blue,
            filecolor=blue,
            urlcolor=blue,
            ]{hyperref}

\include{header}

\author[M. Barakat]{Mohamed Barakat}
\address{Department of mathematics, University of Siegen, 57068 Siegen, Germany}
\email{\href{mailto:Mohamed Barakat <mohamed.barakat@uni-siegen.de>}{mohamed.barakat@uni-siegen.de}}

\author[L. K\"uhne]{Lukas K\"uhne}
\address{Max Planck Institute for Mathematics in the Sciences, Inselstr. 22, 04103, Leipzig, Germany}
\address{Fakult\"at f\"ur Mathematik, Universit\"at Bielefeld, Bielefeld, Germany}
\email{\href{mailto:Lukas Kuehne<lukas.kuehne@math.uni-bielefeld.de>}{lukas.kuehne@math.uni-bielefeld.de}}

\begin{document}

\title[Computing the nonfree locus and Terao's freeness conjecture]{Computing the nonfree locus of the moduli space of arrangements and Terao's freeness conjecture}
\begin{abstract}
  In this paper, we show how to compute using Fitting ideals the nonfree locus of the moduli space of arrangements of a rank $3$ simple matroid, i.e., the subset of all points of the moduli space which parametrize nonfree arrangements.
  Our approach relies on the so-called Ziegler restriction and Yoshinaga's freeness criterion for multiarrangements.
  We use these computations to verify Terao's freeness conjecture for rank $3$ central arrangements with up to $14$ hyperplanes in any characteristic.
\end{abstract}

\thanks{An extended abstract of this paper appeared in the \href{https://doi.org/10.14760/OWR-2021-5}{Oberwolfach workshop report 5/2021} and in the \href{https://fachgruppe-computeralgebra.de/data/CA-Rundbrief/car68.pdf}{Computeralgebra Rundbrief Ausgabe 68.}}

\maketitle


\input{Fitting_content.tex}

\bibliographystyle{myalpha}
\input{Fitting.bbl}

\end{document}

%% file: header.tex
\usepackage[utf8]{inputenc}
\usepackage[english]{babel}
\usepackage[T1]{fontenc}
\usepackage{geometry}                
\usepackage{times}
\usepackage{mathrsfs}
\usepackage{mathtools}
\usepackage{latexsym}
\usepackage{amssymb}
\usepackage{amsthm}
\usepackage{epsfig}
\usepackage{colortbl}
\usepackage[all]{xy}
\usepackage{fancyvrb}
\usepackage{graphicx}
\usepackage[font=small,labelfont=bf]{caption}
\usepackage[dvipsnames]{xcolor}
\usepackage{accents} 
\usepackage{enumerate}

\usepackage{wrapfig}
\usepackage{tikz}
\usetikzlibrary{shapes,arrows,matrix,backgrounds,positioning,plotmarks,calc,patterns,
decorations.shapes,
decorations.fractals,
decorations.markings,
decorations.pathreplacing,
decorations.pathmorphing,
decorations.text
}
\usepackage[colorinlistoftodos,shadow]{todonotes}

\usepackage{bm}
\usepackage{multirow}
\usepackage{mdwlist}
\usepackage{stmaryrd}
\usepackage{mathdots} 

\usepackage[toc,page]{appendix}
\usepackage{float}

\usepackage{bigfoot}

\usepackage[sort&compress,capitalise]{cleveref}

\crefname{exmp}{Example}{Examples}

\usepackage[linesnumbered,commentsnumbered,ruled,vlined]{algorithm2e}

\SetCommentSty{mycommfont}

\newtheoremstyle{mytheoremstyle} 
    {5pt}                    
    {5pt}                    
    {\itshape}                   
    {\parindent}                           
    {\bf}                   
    {.}                          
    {.5em}                       
    {}  

\theoremstyle{mytheoremstyle}

\newtheorem{theorem}{Theorem}[section]

\newtheorem{conj}[theorem]{Conjecture}

\newtheorem{prop}[theorem]{Proposition}
\newtheorem{coro}[theorem]{Corollary}

\newtheoremstyle{mytdefintionstyle} 
    {5pt}                    
    {5pt}                    
    {\rm}                   
    {\parindent}                           
    {\bf}                   
    {.}                          
    {.5em}                       
    {}  

\theoremstyle{remark}
\newtheorem{rmrk}[theorem]{Remark}

\theoremstyle{mytdefintionstyle}
\newtheorem{defn}[theorem]{Definition}
\newtheorem{exmp}[theorem]{Example}

\newtheoremstyle{exmp_contd}
    {5pt}                    
    {5pt}                    
    {\rm}                   
    {\parindent}                           
    {\bf}                   
    {.}                          
    {.5em}                       
    {\thmname{#1}\ \thmnumber{ #2}\thmnote{#3}\ (continued)}  
\theoremstyle{exmp_contd}

\input{pre.tex}

\definecolor{darkgreen}{rgb}{0.008,0.617,0.067}
\definecolor{brown}{rgb}{0.6,0.4,0.2}

\newif\ifjournalversion

%% file: pre.tex
\newcommand{\Am}{(\A,m)}

\DeclareMathOperator{\coker}{coker}

\DeclareMathOperator{\Hom}{Hom}

\DeclareMathOperator{\Spec}{Spec}

\newcommand\p{\mathfrak{p}}

\renewcommand\AA{\mathbb{A}}
\newcommand\Slice{\Sigma}

\newcommand\A{\mathcal{A}}

\newcommand{\Q}{\mathbb{Q}}
\newcommand{\CC}{\mathbb{C}}

\newcommand{\Z}{\mathbb{Z}}
\newcommand\N{\mathbb{N}}
\renewcommand\phi{\varphi}

\newcommand{\RR}{\mathbb{R}}

\DeclareMathOperator\ch{char}

\DeclareMathOperator{\Der}{Der}

\definecolor{darkgray}{rgb}{0.3,0.3,0.3}
\definecolor{LightGray}{gray}{0.9}

\pgfarrowsdeclarecombine[\pgflinewidth]
  {doublestealth}{doublestealth}{stealth'}{stealth'}{stealth'}{stealth'}

%% file: Fitting_content.tex
\section{Introduction}

A (central) arrangement of hyperplanes $\A = \{ H_1, \ldots, H_n \}$ is a finite collection of hyperplanes in a vector space $V$ of dimension $r$ over a field $k$ containing the origin.
Denote by $S$ the polynomial ring $k[x_1,\ldots,x_{r}]$ where $r = \dim V$.
For each hyperplane $H \in \A$ we can fix a linear polynomial $\alpha_H \in S$ such that $H = V(\alpha_H)$ is the vanishing locus of $\alpha_H$.
Then $\A = V(\prod_{H \in \A} \alpha_H)$.

One of the most studied algebro-geometric invariants of $\A$ is the graded $S$-module of \textbf{logarithmic derivations} $D(\A)$ defined as
\[D(\A) \coloneqq \left\{ \theta\in \Der(S) \middle|\, \theta(\alpha_H) \in \alpha_H S \mbox{ for all }H\in \A \right\} \mbox{,}
\]
where $\Der(S) = \left\{ \theta = \sum_{i = 1}^r \theta_i \frac{\partial}{\partial x_i} \middle|\, \theta_i \in S \right\} \cong S^r$ is the module of all derivations on $S$.
If~$D(\A)$ is a free $S$-module, $\A$ is called a \textbf{free arrangement}. 

\begin{conj}[Terao's freeness conjecture]
	The freeness of an arrangement $\A$ defined over a field $k$ only depends on the characteristic of the field and the intersection lattice $L(\A)$ of hyperplanes, which is isomorphic to the lattice of flats of the underlying matroid.
\end{conj}

Recently, Dimca, Ibadula, and Macinic confirmed Terao's conjecture for arrangements in~$\CC^3$ with up to 13 hyperplanes \cite{DIM19}.
In joint work with Behrends, Jefferson, and Leuner, we confirmed Terao's conjecture for rank $3$ arrangements with exactly $14$ hyperplanes in arbitrary characteristic~\cite{BBJKL}.

The main application of the tools we develop in this paper is a common generalization of the two aforementioned results:
\begin{theorem}\label{thm:terao}
	Terao's freeness conjecture is true for rank $3$ arrangements with up to $14$ hyperplanes in any characteristic.
\end{theorem}
In \Cref{sec:matroids} we use our database developed in \cite{BBJKL} to show that there are $9$ matroids left to consider for the proof of \Cref{thm:terao}, which we defer to \Cref{sec:proof}.
In \Cref{sec:database} we list the database keys of these $9$ exceptional cases.

We now describe the main tools we develop in this paper.
Let $\mathcal{R}(M)$ denote the space of all matrix representations of a matroid $M$ (over an arbitrary field) (cf.~\Cref{sec:RM}).
Since $\mathcal{R}(M)$ is too big to compute with, we consider closed subsets $\Slice \subseteq \mathcal{R}(M)$ which
\begin{itemize}
  \item have small dimension compared to $\mathcal{R}(M)$,
  \item but still contain a representative of each equivalence class of arrangements representing $M$.
\end{itemize}
In \Cref{sec:slice} we call any such $\Slice$ a \textbf{representation slice} of the matroid $M$.
In \Cref{sec:computing_slice} we show how to compute a representation slice $\Slice$.

For rank $3$ simple matroids, we then consider in \Cref{sec:NFL} the so-called \textbf{nonfree locus} $\operatorname{NFL}_\Slice(M)$ of a representation slice $\Slice$, which corresponds to the subset of nonfree arrangements in $\Sigma$.
We prove in \Cref{thm:main} that $\operatorname{NFL}_\Slice(M)$ is the vanishing locus of a Fitting ideal of a specific matrix, which relies on the embedding of $\Sigma \subseteq \mathcal{R}(M)$.
We construct this matrix in several steps.
First we describe in \Cref{coro:psi} the graded module of logarithmic derivations of a multiarrangement as the kernel of a morphism between \emph{free} graded modules.
The specific multiarrangement we are interested in is the so-called Ziegler restriction which we briefly recall in \Cref{sec:Ziegler}.
There we also recall Yoshinaga's criterion, which states that the nonfree locus can be detected in a specific degree of the above mentioned morphism of free graded modules.
We therefore describe in \Cref{sec:homogeneous_part} how to compute homogeneous parts of morphisms of free graded modules.

The feasibility of the computation of this Fitting ideal relies on two techniques:
\begin{itemize}
  \item An embedding of $\Slice$ in a smaller affine space as a quasi-affine set which we describe in \Cref{sec:smaller_embedding}.
  \item In \Cref{sec:Fitting} we recall that the Fitting ideal of a matrix is an invariant of its cokernel module.
    Moreover in \Cref{sec:ByASmallerPresentation} we describe several heuristics how to find a smaller presentation matrix for this module.
\end{itemize}

We end the paper with further examples and a computationally motivated conjecture in \Cref{sec:conj}.

\section*{Acknowledgments}

We thank Takuro Abe for inspiring discussions.

\section{Matroids and their representation spaces}\label{sec:representations}

This section is an elaboration of \cite[Section~4]{BBJKL}.

\begin{defn}
  A \textbf{matroid} $M$ is a pair $(E,\mathcal B)$, where $E$ is a finite \textbf{ground set} and $\mathcal B$ is a nonempty set of subsets of $E$, called \textbf{bases}, such that for any two bases $B,B'\in \mathcal{B}$ with $i\in B\setminus B'$ there exists $j\in B'$ with $B\setminus\{i\}\cup\{j\}\in \mathcal{B}$.
  
  The \textbf{size} of $M$ is the size of the ground set $E$ and the common size of all bases is called the \textbf{rank} of $M$.
  The matroid $M$ is called \textbf{simple} if each pair of elements of $E$ is contained in at least one basis.
\end{defn}

Let $M = (\{1, \ldots, n\}, \mathcal{B})$ be a simple rank $r$ matroid.
A \textbf{representation} of $M$ over the field $k$ is a matrix $P \in k^{r \times n}$ such that
\begin{align} \label{eq:representation}
  \det P_B \neq 0 \iff B \in \mathcal{B} \mbox{,}
\end{align}
where $P_B$ is the $r \times r$-submatrix consisting of the columns index by $B$.
The kernels of the linear forms given by the columns of $A$ define an arrangement $\A$ of $n$ hyperplanes in $k^r$ with an intersection lattice isomorphic to the lattice of flats of the matroid $M$.

\subsection{The space of matrix representations of a matroid} \label{sec:RM}

Condition \eqref{eq:representation} defines an ideal $I'$ in the ring $A' = A[d]$, where
\[
  A = \Z[p_{ij} \mid i = 1, \ldots, r, j = 1, \ldots, n ]
\]
given by
\begin{align} \label{eq:I'} \tag{$I'$}
  I' = \left\langle \det(P_N) \mid N \subseteq E \mbox{ not a basis}, |N|=r  \right\rangle + \left\langle 1 -d \prod_{B \in \mathcal{B}(M)} \det(P_B) \right\rangle  \unlhd A[d] \mbox{,}
\end{align}
where $P = (p_{ij}) \in A^{r \times n}$.

The (possibly empty) \textbf{space of matrix representations}\footnote{over some unspecified field $k$} of a matroid $M$ is an \emph{affine variety}, namely the vanishing locus
\begin{align} \label{eq:affine}
  \mathcal{R}(M) \coloneqq V(I') \subseteq \AA^{rn + 1} = \Spec A[d] \twoheadrightarrow \Spec \Z \mbox{.}
\end{align}

The matroid $M$ is representable (over some field $k$) if and only if $1 \notin I'$, which is equivalent to the reduced Gröbner basis (over $\Z$) of $I'$ being equal to $\{1\}$.
This is basically the algorithm suggested in \cite{Oxley2011}.

\subsection{A quasi-affine embedding of the space of matrix representations}
However, it is computationally more efficient to represent $\mathcal{R}(M)$ as a
 locally closed set
\begin{align} \label{eq:quasi-affine}
  \mathcal{R}(M) \cong V(\widetilde{I}) \setminus V(\widetilde{J}) \subseteq \AA^{rn} = \Spec A \twoheadrightarrow \Spec \Z \mbox{,}
\end{align}
where
\begin{align}
  \widetilde{I} &= \sum \left\langle \det P_N \mid N \subseteq E \mbox{ not a basis}, |N|=r  \right\rangle, \label{eq:I_tilde} \tag{$\widetilde{I}$} \\
  \widetilde{J} &= \prod \langle \det P_B \mid B \in \mathcal{B}(M) \rangle \label{eq:J_tilde} \tag{$\widetilde{J}$} \mbox{.}
\end{align}
In particular, $\widetilde{J}$ is a principal ideal\footnote{One can recover the original affine description of $\mathcal{R}(M)$ from the latter quasi-affine description by passing to the Rabinowitsch cover \cite[Example 6.1]{BL_Chevalley}}.
It follows that
\begin{align*}
  \mathcal{R}(M) &\cong V(\widetilde{I}) \setminus V(\widetilde{J})
\end{align*}
and $M$ is representable (over some field $k$) if and only if $\det(P_B) \notin \sqrt{\widetilde{I}}$ for all $B \in \mathcal{B}(M)$.
It is well-known that the radical ideal membership problem can be replaced by a saturated ideal membership problem, which we will utilize towards the end of this section.

\subsection{A representation slices of a matroid} \label{sec:slice}

The algebraic group $\operatorname{GL}_r \times (\operatorname{GL}_1 \wr\, S_n)$ acts on $\AA^{rn}$ by
\[
  (g, ((\lambda_1, \ldots, \lambda_n), \pi)) P = g P h^{-1} \mbox{,}
\]
where $g \in \operatorname{GL}_r$ and $h = (h_{ij}) = (\delta_{i,\pi(j)} \lambda_j)$ is the monomial matrix of
\[
  ((\lambda_1, \ldots, \lambda_n), \pi) \in \operatorname{GL}_1 \wr\, S_n \coloneqq (\operatorname{GL}_1 \times \cdots \operatorname{GL}_1) \rtimes S_n \mbox{.}
\]
This action induces an action on the invariant subscheme $\mathcal{R}(M)$, the orbits of which are the equivalence classes of matrix representations of the matroid $M$.
The moduli space of representations of $M$ is thus the global quotient stack
\[
  \mathcal{M}(M) \coloneqq \mathcal{R}(M) / \left( \operatorname{GL}_r \times (\operatorname{GL}_1 \wr\, S_n) \right) \mbox{.}
\]
Instead of constructing the moduli space $\mathcal{M}(M)$ it suffices to consider a closed subset
\begin{align*}
  \Slice \subseteq \mathcal{R}(M)
\end{align*}
which intersects each orbit of the action of $\operatorname{GL}_r \times (\operatorname{GL}_1 \wr\, S_n)$ on $\mathcal{R}(M)$, i.e., where the projection $\pi:\mathcal{R}(M) \twoheadrightarrow \mathcal{M}(M)$ still restricts to a projection $\pi_{|\Slice}: \Slice \twoheadrightarrow \mathcal{M}(M)$.
We call any such $\Slice$ a \textbf{representation slice}\footnote{We do not call it a representation section since we do not require it to intersect each orbit in exactly one point. However, in many instances the representation slice we found is indeed a representation section.} of the matroid $M$.

\subsection{Computing a representation slice} \label{sec:computing_slice}

Representation slices that can be embedded in affine spaces of smaller dimension are computationally favorable.
In the ideal case the dimension of $\Slice$ should be equal to the dimension of the moduli space $\mathcal{M}(M)$.
This happens when $\Slice$ intersects each orbit in finitely many points.

We construct such a slice by fixing certain values of the matrix $P$ to $0$ or $1$ as described in~\cite[p. 184]{Oxley2011}:
Firstly, we choose a basis $B\in \mathcal{B}(M)$ and fix the corresponding submatrix $P_B$ to be the unit matrix.
Without loss of generality we can assume $B=\{1,\dots,r\}$.
Secondly, we consider the fundamental circuits with respect to this basis~$B$, i.e., for each $i \in E \setminus B$ let $C(i,B)$ be the unique circuit of the matroid $M$ contained in $B\cup \{i\}$.
The entries of $P$ in the column $i \in E \setminus B$ which do not appear in $C(i,B)$ can be fixed to $0$.
Lastly, the first nonzero entry in every column and the first nonzero entry in every row of $P$ can be taken as $1$ by column and row scaling respectively.
We have added this algorithm to $\mathtt{alcove}$ \cite{alcove}.
This amounts to adding elements of the form $p_{ij}$ or $p_{i'j'} - 1$ to the ideal $\widetilde{I}$ defined in \eqref{eq:I_tilde} yielding the larger ideal
\begin{align} \label{eq:I_Sigma} \tag{$I^\Slice$}
  I^\Slice \unlhd R
\end{align}
with
\begin{align*}
   \Slice = V(I^\Slice) \setminus V(\widetilde{J}) = V(I^\Slice) \setminus V(I^\Slice+\widetilde{J}) \mbox{,}
\end{align*}
where $\widetilde{J}$ is the ideal defined in \eqref{eq:J_tilde}.

The ideal $I^\Slice$ can be replaced by the saturation
\begin{align} \label{eq:I} \tag{$I$}
I \coloneqq I^\Slice : \widetilde{J}^\infty = I^\Slice : \det(P_{B_1})^\infty : \cdots : \det(P_{B_b})^\infty \mbox{,}
\end{align}
when $\mathcal{B}(M) = \{B_1, \ldots, B_b\}$.
Likewise, the ideal $\widetilde{J}$ can be replaced by the ideal
\begin{align}
  J \coloneqq \left\langle \operatorname{NF}_{GB(I)}(\det P_B) \mid B \in \mathcal{B}(M) \right\rangle \mbox{,} \label{eq:J} \tag{$J$}
\end{align}
where $\operatorname{NF}_{GB(I)}(f)$ is the normal form of the polynomial $f \in R$ with respect to the Gröbner basis $GB(I)$ of the ideal $I$.
It follows that
\begin{align} \label{eq:Sigma} \tag{$\Slice$}
  \Slice &= V(I) \setminus V(J) \mbox{.}
\end{align}
and $M$ is representable (over some field $k$) if and only if $1 \notin I$, which is again equivalent to the reduced Gröbner basis (over $\Z$) of $I$ being equal to $\{1\}$.
For the Gröbner basis computations over $\Z$ we used \textsc{Singular} \cite{Singular412} from within the GAP package $\mathtt{ZariskiFrames}$ \cite{ZariskiFrames}, which is part of the $\mathtt{CAP/homalg}$ project \cite{homalg-project,BL,GPSSyntax}.

The final step is to embed the quasi-affine $\Slice \subseteq \mathcal{R}(M) \subseteq \AA^{rn}$ into an affine space of smaller dimension, which we explain in \Cref{sec:smaller_embedding}.

\subsection{The nonfree locus of a representation slice} \label{sec:NFL}

Each point $\p \in \mathcal{R}(M) \cong \Spec A[d] / I'$ corresponds to a rank $r$ arrangement $\A_{(\p)}$ over the residue class field $\kappa(\p) \coloneqq \operatorname{Frac}(A[d] / \p)$ and with an intersection lattice of hyperplanes isomorphic to the lattice of flats of the matroid~$M$.
This allows us to define:

\begin{defn}
The \textbf{nonfree locus} of a closed subset $\Slice \subseteq \mathcal{R}(M)$ is the subset
\[
  \operatorname{NFL}_\Slice(M) \coloneqq \{ \p \in \Slice \mid D(\A_{(\p)}) \mbox{ is nonfree} \} \subseteq \Slice \mbox{.}
\]
\end{defn}

Using any representation slice $\Slice \subseteq \mathcal{R}(M)$ we define the \textbf{nonfree locus of the moduli space} $\mathcal{M}(M)$ as the image
\[
  \operatorname{NFL}(M) \coloneqq \pi_{|\Slice}\left(\operatorname{NFL}_\Slice(M) \right) \mbox{.}
\]
For the purpose of this paper it suffices to work with $\operatorname{NFL}_\Slice(M)$ since $\pi_\Slice$, which we do not need to construct, is surjective onto $\operatorname{NFL}(M)$.

We will show in our main \Cref{thm:main} how to compute $\operatorname{NFL}_\Slice(M) \subseteq \Slice$ for rank $3$ matroids.
In particular, we will see that the nonfree locus $\operatorname{NFL}_\Slice(M)$ is a \emph{closed} subset of $\Slice$ in this case.

\section{The $9$ remaining matroids}\label{sec:matroids}

In~\cite{BBJKL}, we generated all $815\,107$ simple rank $3$ matroids with up to $14$ elements with integrally splitting characteristic polynomial and stored them in a public database~\cite{matroids_split}.
To investigate Terao's freeness conjecture it suffices to consider the matroids that are
\begin{itemize}
	\item representable over some field,
	\item not essentially\footnote{either uniquely representable or uniquely representable in a single characteristic up to Galois isomorphism.} uniquely representable,
	\item not divisionally free, and
	\item are not unbalanced.
\end{itemize}
See~\cite{BBJKL} for the detailed definitions of these properties.
Somewhat surprisingly, it turns out that there are only $9$ rank $3$ integrally splitting matroids of size up to $14$ satisfying all of these conditions.
 There is one matroid of size $9$, one of size $11$, two of size $12$, and five of size $13$.
 This already verifies Terao's freeness conjecture for rank $3$ arrangements with precisely $14$ hyperplanes.

To deduce the more general~\Cref{thm:terao}, we will subsequently investigate the freeness of these $9$ exceptional matroids in detail.
We start by first describing the matroids in the rest of this section.

Additional details of these nine matroids are shown in~\Cref{tbl:matroids}.
For further inspection of the properties of these matroids the reader is invited to retrieve them from our public database~\cite{matroids_split} using the keys shown~\Cref{sec:database}.

\subsection{The matroid $M_9$}\label{sec:matroid_9}
The matroid $M_9$ has size $9$ and is known as the \emph{affine geometry} $AG(2,3)$ (see for instance~\cite[Example 6.2.2]{Oxley2011}).
A representation of $M_9$ over $\Q(\zeta_3)$ is the reflection arrangement of the complex reflection group $G(3,3,3)$.
As such, it has the defining equation \[(x^3-y^3)(x^3-z^3)(y^3-z^3).\]

Based on this matroid, Ziegler already demonstrated that it is crucial to formulate Terao's freeness conjecture for a fixed field:
An arrangement $\A$ over a field $k$ with underlying matroid $M_9$ is free if and only if $\ch k\neq 3$~\cite{Zie90}.

The matroid $M_9$ has the representation slice
\[
  \Sigma_{M_9} = V(a^2-a+1)\subseteq \Spec \Z[a]
\]
which parametrizes all representations of $M_{9}$ given by the matrix:
\begin{scriptsize}\[
	\begin{pmatrix*}[r]
	1 & 0 & 1 & 0 & 1 & 0       & 1       & 1         & 1	 \\
	0 & 1 & 1 & 0 & 0 & 1       & -a+1 & -a+1   & 1   \\
	0 & 0 & 0 & 1 & 1 & -a+1 & a       & 1         & a  
	\end{pmatrix*}\mbox{.}
	\]
\end{scriptsize}
\subsection{The matroid $M_{11}$}\label{sec:matroid_11}
The matroid $M_{11}$ is of size $11$.
It has a representation over $\Q(\sqrt{5})$ whose projectivization consists of the sides of a regular pentagon together with its five diagonals and the line at infinity.
We call this arrangement the \emph{pentagon arrangement}.
It is depicted in~\Cref{fig:pentagon}.
It is known that the pentagon arrangement in characteristic zero is free but not inductively free~\cite[Example~4.59]{TO}.
We will prove that it is free over a field $k$ if and only if $\operatorname{char} k \neq 2$.

The matroid $M_{11}$ has the representation slice
\[
  \Sigma_{M_{11}} = V(a^2-a-1)\subseteq \Spec \Z[a]
\]
which parametrizes all representations of $M_{11}$ given by the matrix:
\begin{scriptsize}\begin{equation}\label{eq:mat11}
	\begin{pmatrix*}[r]
	1 & 0 & 1 & 1   	& 0 & 1 & 1   & 0& 0  	    & 1		& 1 \\
	0 & 1 & 1 & a+1   & 0 & 0 & 0   & 1& 1 		  & -a   & -a \\
	0 & 0 & 0 & 0   	& 1 & 1 & a  & -1& -a+1 & a+1 & a  
	\end{pmatrix*}\mbox{.}
	\end{equation}
\end{scriptsize}
Thus over characteristic $0$, a point in $\Slice_{M_{11}}$ involves the golden ratio.
Projecting $\Slice_{M_{11}}$ to $\Spec \Z$ we found that $M_{11}$ admits a representation in all characteristics.
In characteristic~$5$ for instance, the equation $a^2-a-1$ factors to $(a+2)^2$ which means that $M_{11}$ admits a representation over the prime field $\mathbb{F}_5$ in this characteristic.

\begin{figure}
	\centering
	\includegraphics[width=0.5\linewidth]{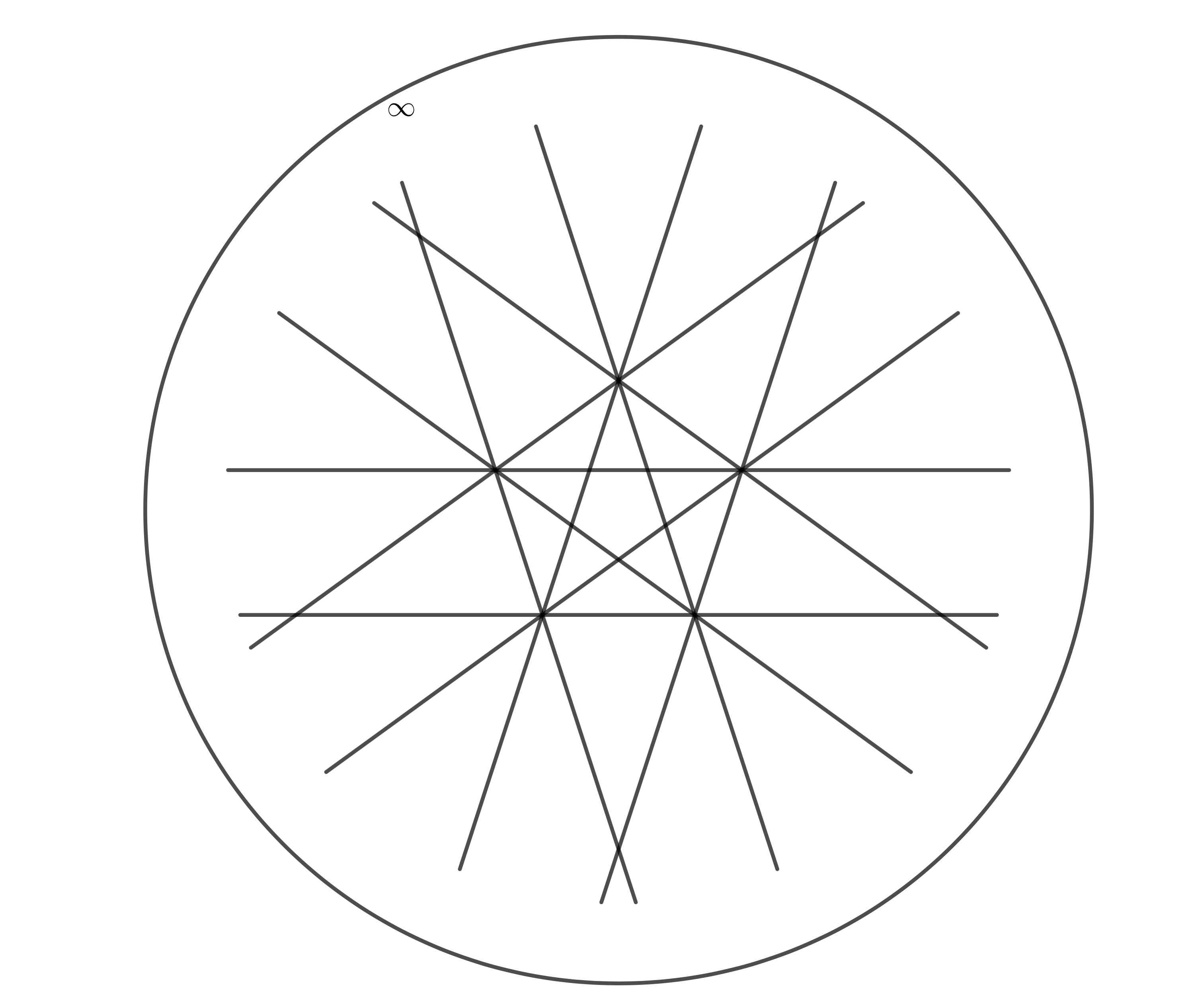}
	\caption{A projectivization of the pentagon representation of $M_{11}$.}
	\label{fig:pentagon}
\end{figure}

\subsection{The matroid $M_{12}^1$}
The matroid $M_{12}^1$ is a matroid of size $12$.
Over $\CC$ is has a representation that is the reflection arrangement over the complex reflection group $G(4,4,3)$ and is given by the following equation:
\[(x^4-y^4)(x^4-z^4)(y^4-z^4).\]
As all reflection arrangements, this arrangement is free over $\CC$.

The matroid $M^1_{12}$ has the representation slice
\[
  \Sigma_{M^1_{12}} = V(a^2-2a+2)\setminus V(2) \setminus V(a-1)\subseteq \Spec \Z[a]
\]
which parametrizes all representations of $M^1_{12}$ given by the matrix:
\begin{scriptsize}\[
	\begin{pmatrix*}[r]
	1 & 0 & 1 & 1   	& 0 & 1         & 1        & 1         & 1 & 0	   & 1        &   1 \\
	0 & 1 & 1 & -1     & 0 & -a+1     & -a+1  & -a+1  & 0 & 1      & 1        & -1\\
	0 & 0 & 0 & 0   	& 1 & 1         & a        & -a+2  & 1 & -a+1 & -a+2 &  a
	\end{pmatrix*}\mbox{.}
	\]\end{scriptsize}

\subsection{The matroid $M_{12}^2$}
The matroid $M_{12}^2$ is a matroid of size $12$.
It is only representable over field extensions of $\mathbb{F}_4$ as it has the representation slice
\[
  \Sigma_{M_{12}^2} = V(2)\setminus V(a) \setminus V(a+1)\subseteq \Spec \Z[a]
\]
which parametrizes all representations of $M^2_{12}$ given by the matrix:
\begin{scriptsize}\[
	\begin{pmatrix*}[r]
	1 & 0 & 1 & 1         & 0 & 1       & 1      & a+1        & 1 & 0	   & a+1 &   1 \\
	0 & 1 & 1 & a^2+1 & 0 & a+1   & a+1  & a^2+1   & 0 & a+1  & a+1 & a^2+1\\
	0 & 0 & 0 & 0         & 1 & 1       & a       & -a          & 1 & 1      & -a   &  a
	\end{pmatrix*}\mbox{.}
	\]\end{scriptsize}

\subsection{The matroid $M_{13}^1$}
The matroid $M_{13}^1$ is of size $13$.
Its representations over $\CC$ were first studied by Abe, Cuntz, Kawanoue, and Nozawa who proved that it is the smallest free arrangement in rank $3$ over $\CC$ that is free but not recursively free~\cite{ACKN16}.
It also has a representation over $\RR$ which is depicted in~\cite[Figure 1]{ACKN16}.

The matroid $M^1_{13}$ has the $1$-dimensional representation slice
\[
  \Sigma_{M^1_{13}} = V(a_1a_2^2-2a_1a_2+a_1-a_2)\setminus V(6a_2^3+4a_1^2-10a_1a_2-a_2^2-a_1-9a_2)\subseteq \Spec \Z[a_1,a_2]
\]
which parametrizes all representations of $M^1_{13}$ given by the matrix:
\begin{scriptsize}\[
	\begin{pmatrix*}[r]
	1 & 0 & 1 & 1 &      1 &  0 & 1 & 1 &      1 &      0 &  0 &  0 &       1 \\
	0 & 1 & 1 & -a_2^2+a_2 & -1 & 0 & 0 & 0 &      0 &      1 &  1 &  1 &       -a_2+1 \\
	0 & 0 & 0 & 0 &      0 &  a_1 & 1 & a_1a_2-a_1+a_2 & -a_2^2+a_2 & -1 & a_1 & -a_1a_2+a_1-a_2 & a_2
	\end{pmatrix*}\mbox{.}
	\]\end{scriptsize}

\subsection{The remaining four matroids}
There are four more matroids of size $13$ which we denote by $M_{13}^2,\dots,M_{13}^5$.
The prominence of these examples is not known to us.
\begin{description}
\item[$M^2_{13}$] 
The matroid $M^2_{13}$ has the representation slice
\[
  \Sigma_{M^2_{13}} = V(a^2+1)\setminus V(2)\subseteq \Spec \Z[a]
\]
which parametrizes all representations of $M^2_{13}$ given by the matrix:
\begin{scriptsize}\[
	\begin{pmatrix*}[r]
	1 & 0 & 1 & 1  &    1 & 0 & 1 & 1   &   1 & 1 & 1  & 1  & 1  \\
	0 & 1 & 1 & -1 &  -a  & 0 & 0 & 0   &   0 & 1 & -1 & -a & a \\
	0 & 0 & 0 & 0  &    0 & 1 & 1 & 1-a & a+1 & 2 & 2  & 2  & 2
	\end{pmatrix*}\mbox{.}
	\]\end{scriptsize}
\item[$M^3_{13}$] 
The matroid $M^3_{13}$ has the representation slice
\[
  \Sigma_{M^3_{13}} = V(a^2-a-1)\setminus V(2)\setminus V(a)\subseteq \Spec \Z[a]
\]
which parametrizes all representations of $M^3_{13}$ given by the matrix:
\begin{scriptsize}\[
	\begin{pmatrix*}[r]
	1 & 0 & 1 & 1    &  1 & 0 & 1 &  1 &  1  &  0 & 0    &  1   &  1\\
	0 & 1 & 1 & -a+1 & -a & 0 & 0 &  0 &  0  &  1 & 1    & -a+1 & -1\\
	0 & 0 & 0 & 0    &  0 & 1 & 1 & -a & a+1 & -1 & -a-1 & a    &  1
	\end{pmatrix*}\mbox{.}
	\]\end{scriptsize}
\item[$M^4_{13}$] 
The matroid $M^4_{13}$ has the representation slice
\[
  \Sigma_{M^4_{13}} = V(2a^2-2a+1)\setminus V(2a-3)\setminus V(3a-1)\subseteq \Spec \Z[a]
\]
which parametrizes all representations of $M^4_{13}$ given by the matrix:
\begin{scriptsize}\[
	\begin{pmatrix*}[r]
	1 & 0 & 1 & 1     &   1 & 0 & 1 & 1  &   0  & 1 & 1    &    0 & 1 \\
	0 & 1 & 1 & -2a+1 & 1-a & 0 & 0 & 0  &   1  & a & 1    &    1 & 1 \\
	0 & 0 & 0 & 0     &   0 & 1 & 1 & 2a & 2-2a & 1 & 2-2a & 1-2a & 1
	\end{pmatrix*}\mbox{.}
	\]\end{scriptsize}
\item[$M^5_{13}$] 
The matroid $M^5_{13}$ has the representation slice
\[
  \Sigma_{M^5_{13}} = V(2a^2+2a+1) \setminus V(2a+3)\setminus V( a+2 )\subseteq \Spec \Z[a]
\]
which parametrizes all representations of $M^5_{13}$ given by the matrix:
\begin{scriptsize}\[
	\begin{pmatrix*}[r]
	1 & 0 & 1 & 1     &   1 & 0 & 1 & 1  &   1   & 0 &  0 &  1    & 1 \\
	0 & 1 & 1 & -2a-1 & -2a & 0 & 0 & 0  &   0   & 1 &  1 & -2a-1 & -2a-2 \\ 
	0 & 0 & 0 & 0     &   0 & 1 & 1 & -a & -2a-1 & a & -1 & 1     & 1
	\end{pmatrix*}\mbox{.}
	\]\end{scriptsize}
\end{description}

\section{Ziegler's restriction and Yoshinaga's criterion} \label{sec:Ziegler}

A (central) arrangement of hyperplanes $\A$ is a finite collection of hyperplanes in a vector space $V$ of dimension $r$ over a field $k$ containing the origin.

A generalization is a \textbf{multiarrangement}, which is defined to be an arrangement of hyperplanes $\A$ with a multiplicity function $m : \A \rightarrow \Z_{>0}$.
A multiarrangement was first defined by Ziegler in~\cite{Zie89} and is denoted by~$\Am$.
Define $|m|\coloneqq \sum_{H\in\A} m(H)$. 

Denote by $S$ the polynomial ring $k[x_1,\ldots,x_{r}]$ where $r = \dim V$.
For each hyperplane~$H$ we can fix a linear defining equation $\alpha_H \in S$.
The $S$-module $D\Am$ is the \textbf{module of logarithmic derivations} of $\Am$ defined as
\[D \Am \coloneqq \left\{ \theta\in \Der(S) \middle|\, \theta(\alpha_H) \in \alpha_H^{m(H)} S  \mbox{ for all }H\in \A \right\},\]
where $\Der(S) \cong S^r$ is the module of all derivations on $S$.
If $D\Am$ is a free $S$-module, we call $\Am$ a \textbf{free multiarrangement}. 
In the case of a free multiarrangement $\Am$ one can choose a homogeneous basis $\theta_1 , \ldots , \theta_r$ of $D\Am$.
In this case we define $\exp \Am =(\deg \theta_1 , \ldots, \deg \theta_r )$ to be the \textbf{exponents} of $\Am$ where a derivation $\theta\in \Der(S)$ is homogeneous with $\deg \theta = d$ if $\theta(\alpha)$ is a homogeneous polynomial of degree $d$ for any $\alpha\in V^*$.

\begin{defn}[Ziegler restriction]
	Let $\A$ be an arrangement and fix some $H\in \A$.
	The \textbf{restricted arrangement} is defined as $\A^{H}\coloneqq \lbrace H \cap L\mid L\in \A\setminus \lbrace H\rbrace \rbrace$.
	A natural multiplicity function $m^H$ on $\A^H$ arises by counting how often a restricted hyperplane $X\in \A^H$ appears as intersection of hyperplanes in $\A$:
	\[
	m^H(X) \coloneqq |\{ L\in \A\setminus \{H\}\mid H\cap L = X \}|.
	\]
	We call the multiarrangement $(\A^H,m^H)$ the \textbf{Ziegler restriction} of $\A$ to the hyperplane $H$.
\end{defn}

\begin{exmp}\label{ex:pentagon_ziegler}
	Consider again the arrangements realizing the matroid $M_{11}$ parametrized by the matrix~\eqref{eq:mat11} in~\Cref{sec:matroid_11}.
	The Ziegler restrictions onto the hyperplane $\{z=0\}$ are parametrized by the matrix
	\[
		\begin{pmatrix*}[r]
		1&   0&   1&    1&          -1\\
		0& 1&   1&    a+1&      a
		\end{pmatrix*}
	\]
	with multiplicities $(3,3,1,1,2)$ over the representation slice $\Sigma_{M_{11}} = V(a^2-a-1)\subseteq \Spec \Z[a]$.
\end{exmp}

Note that a multiarrangement $\Am$ of rank $2$ is always free for some exponents~$(d_1,d_2)$ with $d_1+d_2 = |m|$~\cite{Zie89}.

Our first technical tool to compute the free locus within a representation slice of a matroid is the following remarkable theorem by Yoshinaga.
\begin{theorem}\cite[Theorem 3.2]{Yos05}\label{thm:yoshinaga} \label{theorem:Yoshinaga}
	Let $\A$ be an arrangement of rank $3$ and assume the characteristic polynomial of $\A$ factors as $\chi_{\A}(t) = (t-1)(t-d_2)(t-d_3)$ for some integers~$d_2\leq d_3$.
	Let $H$ be any hyperplane in $\A$ and assume that the Ziegler restriction $(\A^H,m^H)$ is free with exponents $\exp\Am = (d_1',d_2')$ and $d_1'\leq d_2'$.
	Then it holds that
	\begin{equation}\label{eq:yoshinaga}
	d_2d_3 \geq d_1'd_2'
	\end{equation}
	and $\A$ is free if and only if~\eqref{eq:yoshinaga} holds with equality.
\end{theorem}

\begin{rmrk}
	Originally, \Cref{thm:yoshinaga} was formulated only for fields of characteristic $0$ in~\cite{Yos05}.
	This assumption is however not essential, as it was dropped in the subsequent paper by Abe and Yoshinaga which generalized~\Cref{thm:yoshinaga} to higher dimension~\cite{AY13}.
\end{rmrk}

%

\begin{coro} \label{coro:yoshinaga}
	In the notation of \Cref{thm:yoshinaga} the arrangement $\A$ is free if and only if
	\[
	D(\A^H, m^H)_{d_2 - 1} = 0 \mbox{.}
	\]
\end{coro}
\begin{proof}
	Since $\A$ is an arrangement of rank $3$, the Ziegler restriction $(\A^H,m^H)$ is free with exponents $\exp\Am = (d_1',d_2')$ and $d_1'\leq d_2'$.
	By definition of the Ziegler restriction we have $d_2+d_3=|\A|-1=d_1'+d_2'$.
	
	\Cref{thm:yoshinaga} therefore implies $d_2\geq d_1'$ and $\A$ is free if and only if $d_2=d_1'$.
	Thus, $\A$ is free if and only if the lowest degree derivation in $D(\A^H, m^H)$ has degree $d_2$.
\end{proof}

\section{Freeness of arrangements and multiarrangements over a field} \label{sec:multiarrangements}

The algorithm we use to compute $D\Am$ is a direct translation into the language of Gröbner bases of the following Proposition (cf.~\cite[Sec.~6.1]{BC_Coxeter}).
\begin{prop} \label{prop:kernel}
  $D\Am$ is the kernel of the coproduct morphism
  \[
    \mathrm{d}^{\Am} \coloneqq (\pi_{\alpha_H^{m(H)}} \circ \mathrm{d} \alpha_H)_{H \in \A}: \Der(S) \xrightarrow{
    \left(\begin{array}{cccc}
    \frac{\partial \alpha_{H_1}}{\partial x_1} & \cdots & \cdots & \frac{\partial \alpha_{H_n}}{\partial x_1} \\
    \vdots & & & \vdots \\
    \frac{\partial \alpha_{H_1}}{\partial x_r} & \cdots & \cdots & \frac{\partial \alpha_{H_n}}{\partial x_r}
    \end{array}\right)
    } \bigoplus_{H \in \A} S / \alpha_H^{m(H)} S \mbox{,}
  \]
  where $\mathrm{d}f : \Der(S) \to S, \theta \mapsto \theta(f) = \sum_{i=1}^r \theta_i \frac{\partial f}{\partial x_i}$ is the exterior derivative of $f \in S$ and $\pi_f: S \to S / f S$ is the canonical projection.
\end{prop}
\begin{proof}
  The kernel of the universal morphism coincides with the intersection of the kernels of the individual maps $\pi_{\alpha_H^{m(H)}} \circ \mathrm{d} \alpha_H: \Der(S) \to S / \alpha_H^{m(H)} S$.
  The latter was the definition of $D\Am$.
\end{proof}

\begin{coro} \label{coro:psi}
  $D\Am$ is isomorphic to the kernel of a morphism $\psi^{\Am}$ of free graded modules of finite rank.
  More precisely, $D\Am$ is isomorphic to the pullback $K$ in the pullback diagram
  \begin{center}
  \begin{tikzpicture}
    \node(DerA) {$K$};
    \node(SumT) at (4,0) {$\bigoplus_{H \in \A} S(-m(H))$};
    \node(DerS) at (0,-2) {$\Der(S)$};
    \node(SumS) at (4,-2) {$\bigoplus_{H \in \A} S$};
    \draw[right hook-stealth'] (SumT) -- node[right]{$\mu$} (SumS);
    \draw[-stealth'] (DerS) -- node[above]{$\jmath$} (SumS);
    \draw[right hook-stealth'] (DerA) -- node[right]{$\iota$} (DerS);
    \draw[-stealth'] (DerA) -- (SumT);
  \end{tikzpicture}
  \end{center}
  where
  \[
    \jmath:\Der(S)
    \xrightarrow{
    \begin{pmatrix}
    \frac{\partial \alpha_{H_1}}{\partial x_1} & \cdots & \frac{\partial \alpha_{H_n}}{\partial x_1} \\
    \vdots & & \vdots \\
    \frac{\partial \alpha_{H_1}}{\partial x_r} & \cdots & \frac{\partial \alpha_{H_n}}{\partial x_r}
    \end{pmatrix}
    }
    \bigoplus_{H \in \A} S
  \]
  is given by the Jacobi matrix and
  \[
    \mu:
    \bigoplus_{H \in \A} S(-m(H))
    \xhookrightarrow{
    \begin{pmatrix}
\alpha_{H_1}^{m(H_1)} & 0 & \cdots & 0 \\
    0 & \alpha_{H_2}^{m(H_2)} & \ddots & \vdots \\
    \vdots & \ddots & \ddots & 0 \\
    0 & \cdots & 0 & \alpha_{H_n}^{m(H_n)}    \end{pmatrix}
    }
    \bigoplus_{H \in \A} S
  \]
  is an embedding.
  The pullback $K$ (which is isomorphic to $D\Am$) can in turn be computed as the kernel of the coproduct morphism
  \[
    \psi^{\Am} \coloneqq \left(\begin{array}{c}
    \jmath \\
    \hline
    \mu
    \end{array}\right):
    \Der(S) \oplus \bigoplus_{H \in \A} S(-m(H))
    \to
    \bigoplus_{H \in \A} S \mbox{.}
  \]
\end{coro}

\begin{exmp}
	We again consider the matroid $M_{11}$.
	The morphism $\psi^{\Am}$ of the Ziegler restriction described in~\Cref{ex:pentagon_ziegler} over the representation slice $\Sigma_{M_{11}} = V(a^2-a-1)$ is given by the matrix
	\begin{scriptsize}	\begin{equation}\label{eq:mat11_matrix}
		\begin{pmatrix}
		1&   0&   1&    1&          -1\\
		0&   1&   1&    a+1&       a\\
		x^3&0&   0&    0&          0\\
		0&   y^3&0&    0&          0\\
		0&   0&   x+y&0&          0\\
		0&   0&   0&    x+(a+1)y&0\\
		0&   0&   0&    0&          (x-ay)^2
		\end{pmatrix},
	\end{equation}\end{scriptsize}
	as a map $S^2\oplus S(-3)^2\oplus S(-1)^2\oplus S(-2)\to S^5$ where\footnote{$\deg a = 0$ and $\deg x = \deg y = 1$.} $S=\Z[a]/(a^2-a-1)\left[x,y\right]$.
\end{exmp}

\Cref{coro:psi} is a special case of the following result valid in any Abelian category:

\begin{prop} \label{prop:Abelian}
  Let $\iota$ be the (necessarily monic) pullback of a monic $\mu: T \hookrightarrow D$ along a morphism $\jmath: S \to D$ in an Abelian category.
  Further consider the cokernel projection $\varepsilon: D \to C$ and the kernel embedding $\kappa$ of $\varepsilon \circ \jmath$.
  \begin{center}
  \begin{tikzpicture}
    \node(DerA) {$K$};
    \node(SumT) at (4,0) {$T$};
    \node(DerS) at (0,-2) {$S$};
    \node(SumS) at (4,-2) {$D$};
    \node(coker) at (4,-4) {$C = \coker \mu$};
    \node(K') at (-4,0) {$K'$};
    \draw[right hook-stealth'] (SumT) -- node[right]{$\mu$} (SumS);
    \draw[-stealth'] (DerS) -- node[above]{$\jmath$} (SumS);
    \draw[right hook-stealth'] (DerA) -- node[right]{$\iota$} (DerS);
    \draw[-stealth'] (DerA) -- node[above]{$\lambda$} (SumT);
    \draw[-doublestealth] (SumS) -- node[right]{$\varepsilon$} (coker);
    \draw[-stealth'] (DerS) -- node[below left]{$\varepsilon \circ \jmath$} (coker);
    \draw[right hook-stealth'] (K') -- node[below left]{$\kappa$} (DerS);
  \end{tikzpicture}
  \end{center}
  Then $\kappa: K' \hookrightarrow S$ and $\iota: K \hookrightarrow S$ define the same subobject in $S$, i.e., these two monos are mutually dominating.
\end{prop}
\begin{proof}
  We need to construct a unique lift of $\iota$ along $\kappa$ and vice versa, necessarily unique since both morphisms are monos.
  First note that $\iota$ is monic since pullbacks of a monos are monos in any category.
  The morphism $\iota$ lifts  along the kernel embedding $\kappa = \ker(\varepsilon \circ \jmath)$ since $(\varepsilon \circ \jmath) \circ \iota = \underbrace{\varepsilon \circ \mu}_{=0} \circ \lambda = 0$.
  Conversely, the unique lift of $\kappa$ along $\iota$ can constructed as follows:
  First construct the unique kernel lift $\chi: K' \to T$ of $\jmath \circ \kappa: K' \to D$ along the kernel embedding $\mu: T \hookrightarrow D$.
  The desired lift $K' \to K$ of $\kappa$ along $\iota$ is now the universal morphism of the pullback.
\end{proof}

In the notebook \cite{ImageOfPullback} we demonstrate a completely mechanical proof of a generalization of \Cref{prop:Abelian} following Posur's impressive paper \cite{posur2021free}.

\begin{proof}[Proof of \Cref{coro:psi}]
  $D\Am$ is by \Cref{prop:kernel} the kernel of the composition $\varepsilon \circ \jmath$, where $\varepsilon$ is the cokernel projection of $\mu$.
  The first isomorphism in
  \[
    D\Am \cong K \cong \ker \psi^{\Am}
  \]
  is now nothing but \Cref{prop:Abelian} and the second is the well-known fact stated in \Cref{coro:psi}.
\end{proof}

Finally, one computes the kernel of the morphism $\psi^{\Am}$ in the category of free graded modules by computing syzygies.

\section{Fitting ideals}\label{sec:Fitting}

\begin{defn} \label{defn:Fitting}
  Let $A$ be a commutative nonzero unital ring and $\phi: U \to W$ a morphism of free $A$-modules of finite rank.
  After choosing sets of free generators for $U$ and $W$ one can identify $\phi$ with a matrix in $A^{\mathrm{rk}_A U \times \mathrm{rk}_A W}$.
  For $\ell \in \Z$ set $m \coloneqq \mathrm{rk}_A W - \ell$ and define the $\ell$-th \textbf{Fitting ideal}
  \[
    \underbrace{\operatorname{Fitt}_\ell \phi}_{\unlhd A} \coloneqq
    {\small
    \begin{cases}
      A, & m \leq 0,
      \\
      \text{the ideal generated by all $m \times m$ minors of $\phi$}, & 0 < m \leq \min\{ \mathrm{rk}_A U, \mathrm{rk}_A W \},
      \\
      \{0\}, & m > \min \{ \mathrm{rk}_A U, \mathrm{rk}_A W \} \mbox{,}
    \end{cases}}
  \]
  or, equivalently, for $\operatorname{index} \phi \coloneqq \mathrm{rk}_A W - \mathrm{rk}_A U$
  \[
    \underbrace{\operatorname{Fitt}_\ell \phi}_{\unlhd A} \coloneqq
    {\small
    \begin{cases}
      A, & \ell \geq \mathrm{rk}_A W,
      \\
      \text{the ideal generated by all $m \times m$ minors of $\phi$}, & \max\{ \operatorname{index} \phi, 0 \} \leq \ell < \mathrm{rk}_A W,
      \\
      \{0\}, & \ell < \max\{ \operatorname{index} \phi, 0 \} \mbox{.}
    \end{cases}}
  \]
\end{defn}

\begin{theorem}[Fitting's Lemma, {\cite[Cor.-Def.~20.4]{Eis}}] \label{thm:Fitting}
  The $\ell$-th Fitting ideal of $\phi$ only depends on the isomorphism type of $\coker \phi$.
  In particular, it does not depend on the choice of free bases in \Cref{defn:Fitting}.
\end{theorem}

\begin{rmrk} \label{rmrk:Fitting}
  In particular, in order to compute $\operatorname{Fitt}_\ell \phi$ one might pass from the matrix $\phi$ over $A$ (interpreted as a free presentation of $\coker \phi$) to another matrix $\phi'$, preferably of smaller dimensions such that either
  \begin{enumerate}
    \item $\coker \phi' \cong \coker \phi$ and hence, by \Cref{thm:Fitting}
      \[
        \operatorname{Fitt}_\ell \phi = \operatorname{Fitt}_\ell \phi'
      \]
    \item or $\coker \phi \cong \coker \phi' \oplus A^{\oplus z}$ and hence, by \Cref{defn:Fitting}
      \[
        \operatorname{Fitt}_\ell \phi = \operatorname{Fitt}_{\ell-z} \phi' \mbox{.}
      \]
  \end{enumerate}
  These are the two major tricks which allow us to compute Fitting ideals.
  In \Cref{sec:ByASmallerPresentation} we describe several heuristics for finding smaller presentations of finitely presented modules over computable rings.
\end{rmrk}

\begin{rmrk}
  Let $\phi: U \to W$ be a morphism of finite dimensional vector spaces over some field $k$.
  Then the following statements are equivalent for $c \in \Z$:
  \begin{enumerate}
    \item $\dim_k \ker \phi \leq c$;
    \item $m \coloneqq \dim_k U - c \leq \mathrm{rank}_k \phi$;
    \item There exists a nonzero $m \times m$ minor of $\phi$;
    \item $\mathrm{Fitt}_\ell \phi \neq 0$ for
      \[
        \ell \coloneqq \dim_k W - m = c + \operatorname{index} \phi \mbox{,}
      \]
      where $\operatorname{index} \phi \coloneqq \dim_k W - \dim_k U$.
  \end{enumerate}
\end{rmrk}
\begin{proof}
  This follows from the dimension formula $\dim_k \ker \phi + \mathrm{rank}_k \phi = \dim_k U$ and the definition of the Fitting ideals.
\end{proof}

This remark implies:
\begin{coro} \label{coro:Fitt}
  Let $A$ be a commutative nonzero unital ring, $\phi: U \to W$ a morphism of free $A$-modules of finite rank.
  Denote by $\operatorname{index} \phi \coloneqq \mathrm{rk}_A W - \mathrm{rk}_A U$ and for $c \in \Z$ define $\ell \coloneqq c + \operatorname{index} \phi \geq 0$.
  Then
  \[
    V(\mathrm{Fitt}_\ell \phi)
    =
    \{ \p \in \Spec A \mid \dim_{\kappa(\p)} \ker \phi_{\kappa(\p)} > c \} \mbox{,}
  \]
  where $\phi_{\kappa(\p)}: U_{\kappa(\p)} \to W_{\kappa(\p)}$ is the specialization of $\phi$ at $\p$ and $\kappa(\p) \coloneqq \mathrm{Frac}(A/\p)$ is the residue class field of $\p$.
\end{coro}

\section{Proof of \Cref{thm:terao}} \label{sec:proof}

We described in \Cref{sec:computing_slice} how to compute an affine or quasi-affine representation slice.

\begin{defn} \label{defn:phi}
  Let $\Slice = \{ \A_{(\p)} \mid \p \in \Spec A \} \equiv \Spec A$ be an affine representation slice of arrangements representing a rank $3$ matroid with integrally splitting characteristic polynomial $\chi(t) = (t-1)(t-d_2)(t-d_3)$ for some roots $d_2,d_3 \in \Z_{>0}$ with $d_2\le d_3$.
  We view the family $\{ \A_{(\p)} \mid \p \in \Spec A \}$ as an arrangement $\A$ over the ring $A$.
  For a fixed $H \in \A$ define the degree $d_2 - 1$ part of the morphism $\psi^{(\A^H,m^H)}$ (defined in \Cref{coro:psi}) as the morphism
  \[
    \phi^M \coloneqq \psi^{(\A^H,m^H)}_{d_2 - 1}: U \to W
  \]
  of free $A$-modules
  \begin{align*}
    U \coloneqq \left( \Der(S) \oplus \bigoplus_{X \in \A^H} S(-m^H(X)) \right)_{d_2 - 1}, \quad
    W \coloneqq \left( \bigoplus_{X \in \A^H} S \right)_{d_2 - 1} \mbox{.}
  \end{align*}
  Finally we define
  \[
    \operatorname{index} \phi^M \coloneqq \mathrm{rk}_A W - \mathrm{rk}_A U \mbox{.}
  \]
\end{defn}
We show in \Cref{sec:homogeneous_part} how to compute homogeneous parts of morphisms of free graded modules of finite rank.

\begin{theorem} \label{thm:main}
  In the notation of \Cref{defn:phi} the following are equivalent for $\p \in \Spec A$:
  \begin{enumerate}
    \item \label{thm:main.1}
      $\A_{(\p)}$ is not free.
    \item \label{thm:main.2}
      $D(\A_{(\p)}^{H_{(\p)}}, m^{H_{(\p)}})_{d_2 - 1}$ does not vanish.
    \item \label{thm:main.3}
      $\p$ is in the locus $V(\mathrm{Fitt}_\ell \phi^M)$ for $\ell = \operatorname{index} \phi^M$.
  \end{enumerate}
  In particular
  \[
    \operatorname{NFL}_\Slice(M) = V(\mathrm{Fitt}_\ell \phi^M) \mbox{.}
  \]
\end{theorem}
\begin{proof}
  \item[\eqref{thm:main.1}$\Leftrightarrow$\eqref{thm:main.2}:]
    This the statement of \Cref{coro:yoshinaga}.
  \item[\eqref{thm:main.2}$\Leftrightarrow$\eqref{thm:main.3}:]
    This is the combined statement of \Cref{coro:psi} and \Cref{coro:Fitt} for $c = 0$.
\end{proof}

\begin{coro}
  Let $M$ be a rank $3$ simple matroid with a representation slice $\Slice$.
  Then $\operatorname{NFL}_\Slice(M)$ is a \emph{closed} subvariety of $\Slice$.
  In particular, freeness is an open condition. 
\end{coro}

The last statement, restricted to the various fibers of $\Slice \to \Spec \Z$, implies Yuzvinsky's openness result  \cite{Yuz93} for the case of rank $3$ arrangements.

Before discussing the proof of~\Cref{thm:terao} we continue the discussion of the $M_{11}$ matroid underlying the pentagon arrangement.


\begin{exmp}\label{ex:pentagon}
	Since the characteristic polynomial of $M_{11}$ is $\chi(t)=(t-1)(t-5)^2$ the morphism $\phi^M$ is the degree $4$ part of the morphism given by matrix~\eqref{eq:mat11_matrix}.
	The morphism $\phi^M$ is defined by a $25 \times 25$ matrix over $R = \Z[a]/(a^2-a-1)$.
	
	As the index of this morphism is zero, the nonfree locus $\operatorname{NFL}_{\Slice_{M_{11}}}(M_{11})$ is the vanishing locus of the determinant of this square matrix within the representation slice $\Slice_{M_{11}}=V(a^2-a-1)$ of $M_{11}$.
	Using the heuristics described in~\Cref{sec:ByASmallerPresentation} we can replace this matrix by a $1\times 1$ matrix $\phi^M_\mathrm{red} = \begin{pmatrix} 4 \end{pmatrix}\in R^{1 \times 1}$ with $\coker (\phi^M_\mathrm{red}) \cong \coker \phi^M$.
	Hence
	\begin{align*}
		\operatorname{NFL}_{\Slice_{M_{11}}}(M_{11}) &= V(\operatorname{Fitt}_0( \phi^M )) = V(\operatorname{Fitt}_0( \phi^M_\mathrm{red} )) = V\left(\sqrt{\operatorname{Fitt}_0( \phi^M_\mathrm{red})}\right) \\
		&= V(2,a^2-a-1) \subsetneq V(a^2-a-1) = \Slice_{M_{11}} \subset  \Spec \Z[a] \mbox{.}
	\end{align*}
	Therefore, a representation of $M_{11}$ is free if and only if the underlying field is not of characteristic $2$.
	
	There are in fact nonfree representations of $M_{11}$ over $\mathbb{F}_4$.
	A posteriori, this can also be theoretically explained through another freeness criterion of Yoshinaga which applies to arrangements over finite fields~\cite{Yos06}.
	Hence, the ``pentagon'' matroid $M_{11}$ is another example witnessing the dependence of freeness on the underlying field.
\end{exmp}

\begin{proof}[Proof of \Cref{thm:terao}]
	As explained in~\Cref{sec:matroids}, verifying Terao's freeness conjecture for the $9$ exceptional matroids $M_9,M_{11},M_{12}^1,M_{12}^2,M_{13}^1,\ldots,M_{13}^5$ completes the proof.
	
	Analogous to \Cref{ex:pentagon} we computed the free locus of each of these matroids using the techniques of~\Cref{sec:Fitting} and report the results in~\Cref{tbl:matroids}.
	Over a fixed characteristic, the representations of a given matroid turn out to be either all free or all nonfree.
	Notice, however that the freeness depends on the characteristic for representations of $M_9$ and $M_{11}$; all representations are free except in characteristic $3$ and $2$, respectively.
\end{proof}

{\tiny
	\begin{figure}[H]
		\addtolength{\tabcolsep}{0.9pt}
		\def\arraystretch{1.5}{
		\begin{tabular}{lrrrrrrr}
			 $M$& $|M|$ & roots  & $|\mathrm{Aut}_M|$&  the representation slice $\Slice_M$ within $\Spec\Z\left[a \right]$ & $\phi^M$ & $\phi^M_\mathrm{red}$ & $\operatorname{NFL}_{\Sigma_M}(M) \subseteq \Sigma_M$ \\
			\hline
			\hline
			\rowcolor{LightGray}
			$M_9$ & $9$ & $(4,4)$& $432$ & $V(a^2-a+1) \cong \Spec \Z[\zeta_6] = \Spec \Z[\zeta_3]$ & $16 \times 16$ & $1 \times 1$ & $V(3)$ \\
			\hline
			$M_{11}$ &$11$ & $(5,5)$ & $20$ & $V(a^2-a-1) \cong \Spec \Z[\omega],\, \omega = \frac{1+\sqrt{5}}{2}$ & $25 \times 25$ & $1 \times 1$ & $V(2)$\\
						\hline
			\rowcolor{LightGray}
			$M^1_{12}$ & $12$ & $(5,6)$& $192$ & $V(a^2-2a+2)\setminus V(2) \setminus V(a-1)$ & $24 \times 25$ & $8 \times 4$ & $\emptyset$ \\
			\hline
			$M^2_{12}$ &$12$ & $(5,6)$ & $576$ &  $V(2)\setminus V(a) \setminus V(a+1)$ & $24 \times 25$ & $4 \times 4$ & $\Slice_{M_{12}^2} $\\
			\hline
			\rowcolor{LightGray}
			$M^1_{13}$ & $13$ & $(6,6)$& $18$ & $(\star)$ & $36 \times 36$ & $9 \times 2$ & $\emptyset$ \\
			\hline
			$M^2_{13}$ &$13$ & $(6,6)$ & $16$ &  $V(a^2+1)\setminus V(2)$ & $36 \times 36$ & $8 \times 5$ & $\Slice_{M_{13}^2} $\\
			\hline
			\rowcolor{LightGray}
			$M^3_{13}$ & $13$ & $(6,6)$& $8$ & $V(a^2-a-1) \setminus V(2)\setminus V(a)$ & $36 \times 36$ & $2 \times 2$ & $\Slice_{M_{13}^3} $ \\
			\hline
			$M^4_{13}$ &$13$ & $(6,6)$ & $2$ &  $V(2a^2-2a+1)\setminus V(2a-3)\setminus V(3a-1)$ & $36 \times 36$ & $2 \times 1$ & $\emptyset$ \\
			\hline
			\rowcolor{LightGray}
			$M^5_{13}$ & $13$ & $(6,6)$& $2$ & $V(2a^2+2a+1) \setminus V(2a+3)\setminus V( a+2 )$ & $36 \times 36$ & $2 \times 1$ & $\emptyset$
		\end{tabular}}
		\captionof{table}{\rule{0em}{2em}
			The $9$ exceptional matroids together with the two nontrivial roots of its characteristic polynomial.
			We also describe the size of its automorphism group, the representation slice $\Slice_M$, and the nonfree locus $\operatorname{NFL}_{\Sigma_M}(M)$ within $\Slice_M$
			\label{tbl:matroids}}
		\addtolength{\tabcolsep}{-4pt}
	\end{figure}
}

The representation slice $(\star)$ of $M_{13}^1$ is
\[
	\Sigma_{M_{13}^1} = V(a_1a_2^2-2a_1a_2+a_1-a_2)\setminus V(6a_2^3+4a_1^2-10a_1a_2-a_2^2-a_1-9a_2) \subseteq \Spec \Z[a_1,a_2] \mbox{.}
\]

\section{Further examples with more than $14$ lines} \label{sec:conj}

The two exceptional matroids $M_9$ and $M_{12}^1$ underlie the arrangements corresponding to the complex reflections groups $G(3,3,3)$ and $G(4,4,3)$, respectively.
It is known that the reflection arrangements over characteristic $0$ stemming from $G(n,n,3)$ are free for all $n\ge 1$.
We determined the obstruction variety of their underlying matroids for $n<10$ and found the following:
If $n$ is a prime number, a representation is free if and only if the characteristic of the field is not $n$.
If $n$ is not a prime number, the obstruction variety was empty in all considered cases.
Therefore, we pose the following conjecture:

\begin{conj}
	An arrangement  $\A$ over a field $k$ which is combinatorially equivalent to the reflection arrangement $G(n,n,3)$ for some $n\ge 3$ is free if and only if $n$ is not a prime number or if the characteristic of $k$ is not $n$ in the case that $n$ is a prime number.
\end{conj}

\appendix

\section{Embedding of affine sets in smaller affine spaces} \label{sec:smaller_embedding}

The goal of this section is to list some simple heuristics we use to embed the representation slice $\Slice \subseteq \mathcal{R}(M)$ into a smaller affine space.
Since the representation space $\mathcal{R}(M)$ has an affine representation given by $\mathcal{R}(M) \subseteq \AA^{rn + 1}$ in \eqref{eq:affine} and a quasi-affine representation $\mathcal{R}(M) \subseteq \AA^{rn}$ given by \eqref{eq:quasi-affine}, the same holds true for the closed subset $\Slice \subseteq \mathcal{R}(M)$.
Recall that in \eqref{eq:Sigma} we have described
\begin{align*}
   \Slice = V(I) \setminus V(J) \mbox{,}
\end{align*}
with $I, J \unlhd A$.

What we describe now works for any ideal $I$ in a polynomial ring $A = \Z[x_1, \ldots, x_t]$.\footnote{$A=\Z[p_{ij} \mid i = 1, \ldots, r, j = 1, \ldots, n]$ in \Cref{sec:RM}.}
Our goal is to find a subset $\{ z_1, \ldots, z_u \} \subseteq \{ x_1, \ldots, x_t \}$ and an ideal $I_\infty \unlhd A_\infty = \Z[z_1, \ldots, z_u]$, such that the natural morphism $A_\infty \to A / I$ yields an isomorphism $A_\infty / I_\infty \cong A / I$, i.e., to replace the natural projection $\pi: A \twoheadrightarrow A / I$ by an equivalent surjective morphism $\pi_\infty: A \twoheadrightarrow A_\infty / I_\infty$
\begin{center}
\begin{tikzpicture}
  \node (A) {$A$};
  \node (I) at (3,0.75) {$A/I$};
  \node (I') at (3,-0.75) {$A_\infty / I_\infty$};
  
  \draw[-doublestealth] (A) -- node[above]{$\pi$} (I);
  \draw[-doublestealth] (A) -- node[below]{$\pi_\infty$} (I');
  \draw[-stealth'] (I) -- node[right]{\rotatebox{90}{$\sim$}} (I');
  
\end{tikzpicture}
\end{center}
which describes an embedding
\[
  \AA^u \supseteq V(I_\infty) \cong V(I) \hookrightarrow \AA^t \mbox{.}
\]
We achieve this by computing successive isomorphisms $A/I \cong A_1 / I_1 \cong \ldots \cong A_\infty / I_\infty$.
\begin{enumerate}
  \item Let $G$ be a Gröbner basis of $I$.
    We call an indeterminate $x_i$ \textbf{standard} if $\operatorname{NF}_{G}(x_i) = x_i$, where $\operatorname{NF}_{G}(f)$ denotes the normal form of the polynomial $f \in A$ with respect to $G$.
    For all other indeterminates the normal form will only contain terms in the polynomial subring $A_1 = \Z[y_1, \ldots, y_s]$ generated by the standard indeterminates $\{y_1, \ldots, y_s\} \subseteq \{x_1, \ldots, x_t\}$.
    Let $I_1$ denote the kernel of the surjective ring morphism $A_1 \to A/I, y_i \mapsto y_i + I$, the composition of the embedding $A_1 \hookrightarrow A$ and the natural projection $\pi: A \twoheadrightarrow A / I$.
    A Gröbner basis $G_1$ of $I_1$ can be computed using an elimination order\footnote{Or in fact any global monomial order \cite{BL_Elimination}.}.
    Composing the ring morphism
    \[
      \widetilde{\pi}_1: A \to A_1, x_i \mapsto \operatorname{NF}_G(x_i)
    \]
    with the natural surjection $A_1 \twoheadrightarrow A_1 / I$ yields the surjective morphism
    \[
      \pi_1: A \twoheadrightarrow A_1 / I_1, x_i \mapsto \operatorname{NF}_G(x_i) + I
    \]
    which describes the embedding $\AA^s \supseteq V(I_1) \cong V(I) \hookrightarrow \AA^t$.
  \item Inspect all elements of the Gröbner basis $G_1$ of $I_1$ for elements of the form $g = u y_i - f(y_1, \ldots, y_{i-1}, y_{i+1}, \ldots, y_s)$ for some $1 \leq i \leq s$ and $u$ is a unit in $A_1$ (here $u = \pm 1$).
  The ring morphism
  \[
    \widetilde{\psi}_2:
    \begin{cases}
      A_1 &\to A_2 = \Z[y_1, \ldots, y_{i-1}, y_{i+1}, \ldots, y_s], \\
      y_j &\mapsto y_j \quad j \neq i, \\
      y_i &\mapsto f / u \mbox{.}
    \end{cases}
  \]
  yields the isomorphism $A_1 / I_1 \xrightarrow{\sim} A_2 / I_2$ with $I_2 \coloneqq \widetilde{\psi}_2(I_1)$.
  Composing $\widetilde{\pi}_2 \coloneqq \widetilde{\psi}_2 \circ \widetilde{\pi}_1: A \to A_2$ with the natural surjection $A_2 \twoheadrightarrow A_2 / I_2$ we get the surjective morphism $\pi_2: A \twoheadrightarrow A_2 / I_2$ which describes the embedding $\AA^{s-1} \supseteq V(I_2) \cong V(I) \hookrightarrow \AA^t$.
  \item Iterate the last step until no element of the Gröbner basis of the ideal $I_i$ has the above mentioned form. Then set $A_\infty \coloneqq A_i$, $I_\infty \coloneqq I_i$ and $\pi_\infty: A \twoheadrightarrow A_\infty / I_\infty$.
\end{enumerate}
For any other ideal $J \unlhd A$ we can define $J_\infty \coloneqq \widetilde{\pi}_\infty(J)$ and obtain the embedding
\[
  \AA^u \supseteq V(I_\infty) \setminus V(J_\infty) \cong V(I) \setminus V(J) \hookrightarrow \AA^t \mbox{.}
\]

These heuristics are implemented in the package $\mathtt{MatricesForHomalg}$ \cite{homalg-project} and interpreted geometrically in $\mathtt{ZariskiFrames}$ \cite{ZariskiFrames}.

\section{Computing the degree $d$ part of $\phi$} \label{sec:homogeneous_part}

Here we show how to compute homogeneous parts of morphisms of free graded modules of finite rank which we use to compute the degree-$(d_2-1)$-part $\phi^M \coloneqq \left(\psi^{(\A^H,m^H)}\right)_{d_2 - 1}$ of the morphism $\psi^{(\A^H,m^H)}$ considered in \Cref{sec:proof}.

Let $A$ be a commutative ring and $S = A[x_1, \ldots, x_r]$ the free polynomial $A$-algebra, equipped with the standard grading $\deg(x_i) = 1$ and $\deg(a) = 0$ for all $a \in A \setminus \{0\}$.
Consider the category $A^\oplus$ of free $A$-modules of finite rank and the ($A^\oplus$-enriched) category $S^\oplus$ of free \emph{graded} modules of finite rank consisting of the objects
\[
  M = \bigoplus_{i=1}^h S(-m_i) \qquad \mbox{ for } h \in \N, \forall_{i=1}^h m_i \in \Z \mbox{.}
\]
For $d \in \Z$ the degree-$d$-part is the additive functor
\[
  (-)_d \coloneqq \Hom_{S^\oplus}(S(-d), -): S^\oplus \to A^\oplus.
\]
The value of the functor on an object $M = \bigoplus_{i=1}^h S(-m_i) \in S^\oplus$ is given by
\begin{align*}
  M_d
  &=
  \Hom_{S^\oplus}(S(-d), M)
  =
  \Hom_{S^\oplus}\left(S(-d), \bigoplus_{i=1}^h S(-m_i) \right) \\
  &=
  \bigoplus_{i=1}^h \Hom_{S^\oplus}(S(-d), S(-m_i)) = \bigoplus_{i=1}^h S(-m_i)_d = \bigoplus_{i=1}^h S_{d - m_i} \mbox{,}
\end{align*}
where $S_a$ is the degree-$a$-part of the polynomial ring $S$, which is a free $A$-module of rank $\binom{r-1+a}{r-1}$.

The value of the functor on a morphism $\psi: M \to N$ in $S^\oplus$ can be constructed as follows.
For $a \in \Z$ define the column matrix $\iota_a \in S^{\binom{r-1+a}{r-1} \times 1}$ consisting of all degree-$a$ monomials in $S$.
The embedding of the $S$-submodule $\langle M_d \rangle \leq M = \bigoplus_{i=1}^h S(-m_i)$ generated by the degree-$d$-part $M_d$ is given by the block-diagonal matrix
\[
  \iota_{M,d}: \mathrm{diag}\left(\iota_{d-m_1}, \ldots, \iota_{d-m_h}\right): \langle M_d \rangle \hookrightarrow M \mbox{,}
\]
interpreted as a morphism in $S^\oplus$.
The degree-$d$-part of $\psi$ can now be computed by computing the lift $\langle \psi_d \rangle$
\begin{center}
  \begin{tikzpicture}
    \node(Md) {$\langle M_d \rangle$};
    \node(M) at ($(Md)+(2,0)$) {$M$};
    \node(Nd) at ($(Md)+(0,-2)$) {$\langle N_d \rangle$};
    \node(N) at ($(Nd)+(2,0)$) {$N$};
    
    \draw[right hook-stealth'] (Md) -- node[above]{$\iota_{M,d}$} (M);
    \draw[right hook-stealth'] (Nd) -- node[above]{$\iota_{N,d}$} (N);
    \draw[-stealth'] (M) -- node[right]{$\psi$} (N);
    \draw[dotted,-stealth'] (Md) -- node[left]{$\langle \psi_d \rangle$} (Nd);
  \end{tikzpicture}
\end{center}
as a matrix over $S$ (using Gröbner bases over $S$).
Due to degree reasons the matrix $\langle \psi_d \rangle$ is in fact a matrix $\psi_d$ over $A$, which defines the desired morphism $\psi_d: M_d \to N_d$.
Note that the lift is unique as a lift (of the composition $\psi \circ \iota_{M,d}$) along the monomorphism $\iota_{N,d}$.

The above functor is implemented in the $\mathsf{GAP}$-packages $\mathtt{GradedModules}$ \cite{GradedModules} and $\mathtt{IntrinsicGradedModules}$ \cite{IntrinsicGradedModules}. The former is based on $\mathtt{homalg}$ \cite{BL} and the latter is a reimplementation based on $\textsc{Cap}$ \cite{GPSSyntax,PosCCT+arXiv}.

\section{Finding a smaller presentation of a finitely presented module over an affine ring and computing Fitting ideals} \label{sec:ByASmallerPresentation}

Let $A$ be a commutative nonzero unital ring and $\phi: U \to W$ a morphism of free $A$-modules of finite rank.
After choosing free bases of $U$ and $W$ one can identify the morphism $\phi$ with a matrix\footnote{We use the row-convention.} $\phi \in A^{\mathrm{rk}_A U \times \mathrm{rk}_A W}$.
Using \Cref{rmrk:Fitting} we now describe several heuristics which start with the matrix $\phi$ and try to produce a matrix $\phi'$ of smaller dimensions, and where the Fitting ideals of $\phi$ can be computed in terms of the Fitting ideals of $\phi'$.
\begin{enumerate}
  \item \label{heuristics:a}
    The rows of $\phi$ describe a generating set of relations among the generators of the $A$-module $\coker \phi$.
    Hence, the matrix $\phi$ can be replaced with any other $? \times \mathrm{rk}_A W$ matrix $\phi'$ having the same row span (= the $A$-submodule of relations in the free $A$-module $A^{1 \times \mathrm{rk}_A W}$).
    In particular, one can remove from $\phi$ zero rows, duplicate rows, etc.
    More generally, if $A$ is a ring with a Gröbner basis notion as in our application, then one can replace $\phi$ by a reduced Gröbner basis of the rows, at least if this results in a smaller generating system of the row span.
  \item Let $\rho \coloneqq \operatorname{row-syz}(\phi)$ be a matrix of syzygies among the rows of $\phi$.
    If any row-syzygy (i.e., row of $\rho$) has a unit at the $i$-th column, then the $i$-th row of $\phi$ is a linear combination of the remaining rows and can be deleted as in \eqref{heuristics:a}.
  \item A unit $\phi_{i,j} \in A^\times$ means that the $i$-th row expresses the $j$-th generator of $\coker \phi$ as an $A$-linear combination of the remaining generators.
    In this case define $\phi'$ as follows:
    Turn the $i$-th row into the $j$-th unit vector by dividing the $j$-th column by~$\phi_{i,j}$ and then use the resulting $1$ to clean up the rest of the $i$-th row.
    This corresponds to multiplying $\phi$ with the invertible matrix $\chi$ defined as the $\mathrm{rk}_A W \times \mathrm{rk}_A W$ identity matrix with the $j$-th row replaced by
    \[
      \begin{pmatrix}
        -\frac{\phi_{i,1}}{\phi_{i,j}} & \cdots & -\frac{\phi_{i,j-1}}{\phi_{i,j}} & \frac{1}{\phi_{i,j}} & -\frac{\phi_{i,j+1}}{\phi_{i,j}} & \cdots -\frac{\phi_{i,\mathrm{rk}_A W}}{\phi_{i,j}}
      \end{pmatrix} \mbox{.}
    \]
    The $i$-th row of the resulting matrix $\phi \chi$ is indeed the $j$-th unit vector in $A^{1 \times \mathrm{rk}_A W}$.
    Multiplying with $\chi$ can be interpreted as a change of the generating system, where the $j$-th unit vector (now occurring as the $i$-th relation) states that in the new generating system of $\coker \phi \chi \cong \coker \phi$ the $j$-th generator is zero.
    This means that the $j$-th column of $\phi \chi$ (containing the coefficients of this new zero generator) can be deleted.
    The $i$-th row of the resulting matrix is zero and can also be deleted as in \eqref{heuristics:a}.
    Hence, the matrix $\phi'$ defined by deleting the $j$-th column and the $i$-th row in $\phi \chi$ obviously satisfies $\coker \phi' \cong \coker \phi$.
  \item \label{heuristics:b}
    If $\phi$ contains a zero column, then $\coker \phi$ admits a free direct summand $A$.
    More precisely, if the $j$-th column of $\phi$ is zero then $\coker \phi \cong \coker \phi' \oplus A$, where $\phi'$ results from $\phi$ by deleting the zero $j$-th column.
    It follows from \Cref{defn:Fitting} that
    \[
      \operatorname{Fitt}_\ell \phi = \operatorname{Fitt}_{\ell-1} \phi' \mbox{.}
    \]
  \item Let $\Slice \coloneqq \operatorname{col-syz}(\phi)$ be a matrix of syzygies among the columns of $\phi$.
    If any column-syzygy (i.e., column of $\Slice$) has a unit at the $j$-th entry then the $j$-th column of $\phi$ is an $A$-linear combination of the remaining columns.
    This syzygy can be used to replace the $j$-th generator of $\coker \phi$ by a free generator of a  free direct summand $A$, i.e., a new generator for which in the transformed relation matrix $\widetilde{\phi}$ the $j$-th column is zero.
    Then
    \[
      \operatorname{Fitt}_\ell(\phi) = \operatorname{Fitt}_\ell(\widetilde{\phi}) \stackrel{\eqref{heuristics:b}}{=} \operatorname{Fitt}_{\ell-1}(\phi') \mbox{,}
    \]
    where $\phi'$ is the matrix $\phi$ (or $\widetilde{\phi}$) with the $j$-th column deleted.
\end{enumerate}

\section{The $9$ exceptional matroids in the database} \label{sec:database}

	\begin{figure}[H]
	\addtolength{\tabcolsep}{0.9pt}
	\def\arraystretch{1.5}{
		\begin{tabular}{lr}
			$M$ & the key of $M$ in the database\\
			\hline
			\hline
			\rowcolor{LightGray}
			$M_9$ &  \texttt{d8ffc0e083ea534e34556086ac2416ca48cc0483}\\
			\hline
			$M_{11}$ & \texttt{ba024dadc693ecf1dcfffcb1080b17412d3ef456}\\
			\hline
			\rowcolor{LightGray}
			$M^1_{12}$ &  \texttt{ba1ed7e5f970cd1f1f61c6f8da9819b11518a676}\\
			\hline
			$M^2_{12}$ & \texttt{0df50b7b1d5adf05683022f4b5dac3deff13df93}\\
			\hline
			\rowcolor{LightGray}
			$M^1_{13}$ & \texttt{118cb9babc77e406eb53043ac399bf851a012830}\\
			\hline
			$M^2_{13}$ &  \texttt{f6326c481d3f3cdf408d7e1a57beac744611e5b4}\\
			\hline
			\rowcolor{LightGray}
			$M^3_{13}$ &  \texttt{cc857179fb73ff8064a5aa3bc4e225df535310d8}\\
			\hline
			$M^4_{13}$ &  \texttt{8b85726983dbde6bcc794225ce124b1e35271399}\\
			\hline
			\rowcolor{LightGray}
			$M^5_{13}$ & \texttt{3582712ca004ed51fa91726c695b5a991af7c1b5}
	\end{tabular}}
	\captionof{table}{\rule{0em}{2em}
		The hashed keys of the $9$ exceptional matroids in the database~\cite{matroids_split}.
		\label{tbl:database}}
	\addtolength{\tabcolsep}{-4pt}
\end{figure}

%% file: Fitting.bbl
\newcommand{\etalchar}[1]{$^{#1}$}
\def\cprime{$'$} \def\cprime{$'$} \def\cprime{$'$} \def\cprime{$'$}
  \def\cprime{$'$}
\providecommand{\bysame}{\leavevmode\hbox to3em{\hrulefill}\thinspace}
\providecommand{\MR}{\relax\ifhmode\unskip\space\fi MR }
\providecommand{\MRhref}[2]{%
  \href{http://www.ams.org/mathscinet-getitem?mr=#1}{#2}
}
\providecommand{\href}[2]{#2}